\documentclass[twoside]{article}
\usepackage{amsmath,amsthm,amssymb}
\usepackage{newtxtext,newtxmath}
\usepackage[text={12.5cm,18.7cm},centering,paperwidth=16.96cm,paperheight=23.68cm]{geometry}
\usepackage[fontsize=10.4pt]{scrextend}
\usepackage[hyperfootnotes=false,colorlinks=true,allcolors=blue]{hyperref}
\usepackage{algorithm}
\usepackage{array}
\usepackage{listings}
\usepackage{multirow}
\usepackage{tabularx}
\usepackage{xcolor}
\usepackage[noend]{algpseudocode}

\definecolor{codegreen}{rgb}{0,0.6,0}
\definecolor{codegray}{rgb}{0.5,0.5,0.5}
\definecolor{codepurple}{rgb}{0.58,0,0.82}
\definecolor{backcolour}{rgb}{0.95,0.95,0.92}

\lstdefinestyle{mystyle}
{
    backgroundcolor=\color{backcolour},
    commentstyle=\color{codegreen},
    keywordstyle=\color{blue},
    numberstyle=\tiny\color{codegray},
    stringstyle=\color{codepurple},
    basicstyle=\ttfamily\footnotesize,
    breakatwhitespace=false,
    breaklines=true,
    captionpos=b,
    keepspaces=true,
    numbers=left,
    numbersep=5pt,
    showspaces=false,
    showstringspaces=false,
    showtabs=false,
    tabsize=2
}
\def\titlerunning#1{\gdef\titrun{#1}}

\makeatletter
\def\author#1{\gdef\autrun{\def\and{\unskip,}#1}\gdef\@author{#1}}

\makeatother

\def\email#1{E-mail: \href{mailto:#1}{#1}}
\def\subjclass#1{\par\bigskip\noindent\textsl{Mathematics Subject Classification 2020. }#1}
\def\keywords#1{\par\smallskip\noindent\textsl{Keywords. }#1}

\newenvironment{dedication}{\itshape\center}{\par\medskip}

\newtheorem{thm}{Theorem}[section]

\newtheorem{lem}[thm]{Lemma}

\theoremstyle{definition}

\numberwithin{equation}{section}

\begin{document}

\titlerunning{An algorithm for $g$-invariant on unary Hermitian lattices over imaginary quadratic fields}

\title{\textbf{An algorithm for $g$-invariant on unary Hermitian lattices \\ over imaginary quadratic fields}}

\author{Jingbo Liu}

\date{}

\maketitle

\begin{dedication}
Department of Computational, Engineering, and Mathematical Sciences\\
Texas A\&M University-San Antonio, San Antonio, Texas 78224, USA\\
\email{jliu@tamusa.edu}
\end{dedication}


\subjclass{Primary 11E39. Secondary 11Y16, 11Y40.}

\keywords{Waring's problem, unary Hermitian lattices, sums of norms, algorithm.}

\begin{abstract}
Let $E=\mathbb{Q}\big(\sqrt{-d}\big)$ be an imaginary quadratic field for a square-free positive integer $d$, and let $\mathcal{O}$ be its ring of integers.
For each positive integer $m$, let $I_m$ be the free Hermitian lattice over $\mathcal{O}$ with an orthonormal basis, let $\mathfrak{S}_d(1)$ be the set consisting of all positive definite integral unary Hermitian lattices over $\mathcal{O}$ that can be represented by some $I_m$, and let $g_d(1)$ be the least positive integer such that all Hermitian lattices in $\mathfrak{S}_d(1)$ can be uniformly represented by $I_{g_d(1)}$.
The main results of this work provide an algorithm to calculate the explicit form of $\mathfrak{S}_d(1)$ and the exact value of $g_d(1)$ for every imaginary quadratic field $E$, which can be viewed as a natural extension of the Pythagoras number in the lattice setting.
\end{abstract}


\section{Introduction}\label{Sec:Int}
A positive integer $a$ is said to be represented by a quadratic form $f$, (often) written as $a\to f$, provided the Diophantine equation $f(\vec{x})=a$ has an integral solution.
A typical question is to determine those positive integers which can be represented by the form $I_m:=x_1^2+\cdots+x_m^2$ with $m$ a positive integer.
In 1770, Lagrange proved the famous {\sf Four-Square Theorem}: The quadratic form $x_1^2+x_2^2+x_3^2+x_4^2$ can represent all positive integers; such quadratic forms are said to be {\sl universal} in the literature.
Recalling that $7$ cannot be a sum of any three integral squares, it is clear that the sum of four squares is best possible for the universality over $\mathbb{Z}$.

This result has been generalized to several directions, and one important direction is to consider the representation of sums of squares in totally real number fields $F$: A positive definite integral quadratic form over a totally real number field $F$ is said to be universal provided it can represent all totally positive integers in $F$.
In 1928, G\"{o}tzky \cite{fG} proved that $x_1^2+x_2^2+x_3^2+x_4^2$ is universal over $\mathbb{Q}\big(\sqrt{5}\big)$ and later Maa{\ss} \cite{hM} further proved the universality of $x_1^2+x_2^2+x_3^2$ over $\mathbb{Q}\big(\sqrt{5}\big)$, which, however, will not occur in any other field $F$, as Siegel \cite{clS} proved that all totally positive integers in $F$ are sums of integral squares if and only if either $F=\mathbb{Q}$ or $F=\mathbb{Q}\big(\sqrt{5}\big)$.
Recall the {\sl Pythagoras number} $P(\mathcal{O}_F)$ of $F$ is defined as the smallest positive integer $p$ such that every sum of integral squares in $F$ is a sum of $p$ integral squares, where $\mathcal{O}_F$ denotes the ring of integers of $F$.
Noticing Maa{\ss} \cite{hM} and that $1+1+\big(1+\sqrt{5}\big)^2/4$ cannot be a sum of any two integral squares in $F=\mathbb{Q}\big(\sqrt{5}\big)$, one has $P(\mathcal{O}_F)=3$.
For other real quadratic fields $F=\mathbb{Q}\big(\sqrt{d}\big)$ with $d$ a square-free positive integer, it is known $P(\mathcal{O}_F)$ was completely determined by the earlier works of Cohn and Pall, Dzewas, Kneser, Maa{\ss}, and Peters: $P(\mathcal{O}_F)=3$ for $d=2,3$, $P(\mathcal{O}_F)=4$ for $d=6,7$, and $P(\mathcal{O}_F)=5$ for all the remaining cases; see Theorem 3.1 of Kr\'{a}sensk\'{y}, Ra\v{s}ka and Sgallov\'{a} \cite{KRS}.
Some recent important contributions in this direction include Blomer and Kala \cite{BK}, Chan and Icaza \cite{CI}, Kala and Yatsyna \cite{KY1,KY2}, and Kr\'{a}sensk\'{y} and Yatsyna \cite{KY} etc.

In this paper, I shall consider the Hermitian analog of this representation problem over imaginary quadratic fields: Let $E=\mathbb{Q}\big(\sqrt{-d}\big)$ be an imaginary quadratic field with $d$ a square-free positive integer, and denote by $\mathcal{O}$ its ring of integers.
For each positive integer $m$, let $I_m:=x_1\overline{x_1}+\cdots+x_m\overline{x_m}$ be the sum of $m$ norms.
Define the {\sl Pythagoras number} $P(\mathcal{O})$ of $E$ to be the smallest positive integer $p$ such that every sum of integral norms in $E$ is a sum of $p$ integral norms. Kim, Kim and Park \cite[Table 10]{KKP} and Lemma 2.2 in \cite{jL2} completely determined $P(\mathcal{O})$: $P(\mathcal{O})=2$ for $d=1,2,3,7,11$, $P(\mathcal{O})=3$ for $d=5,6,15,19,23,27$, and $P(\mathcal{O})=4$ for all the remaining cases.
Seeing this, a harder problem is addressed here.
In Beli, Chan, Icaza and me \cite{BCIL} and my own work \cite{jL1,jL2}, we studied unary Hermitian lattices--a more general concept than positive integers; as $\mathcal{O}$ needs not to be a principal ideal domain in general, a positive integer $a$ corresponds to a unary Hermitian lattice $\mathcal{O}v$ with the Hermitian map $h$ satisfying $h(v)=a$, but not vice versa.
The standard correspondence between Hermitian forms and free Hermitian lattices implies that $I_m$ corresponds to the free Hermitian lattice $\mathcal{O}v_1+\cdots+\mathcal{O}v_m$ with $I_m$ its associated Hermitian form for the orthonormal basis $\{v_1,\ldots,v_m\}$; so, $I_m$ below denotes both the Hermitian form and its corresponding Hermitian lattice.

A unary Hermitian lattice $L$ is said to be represented by some $I_m$, provided there is an injective linear map from $L$ to $I_m$ preserving the Hermitian maps.
Denote by $\mathfrak{S}_d(1)$ the set consisting of all positive definite integral unary Hermitian lattices over $\mathcal{O}$ which can be represented by some $I_m$, and by $g_d(1)$ the smallest positive integer such that all the lattices in $\mathfrak{S}_d(1)$ can be uniformly represented by $I_{g_d(1)}$.

Over the imaginary quadratic fields $E=\mathbb{Q}\big(\sqrt{-d}\big)$ with class number $1$, {\sl i.e.}, those fields by $d=1,2,3,7,11,19,43,67,163$, each positive definite integral unary Hermitian lattice is of the form $L=\mathcal{O}v$, where $h(v)$ is a positive integer.
$L$ is represented by $I_m$ if and only if $h(v)$ can be written as a sum of $m$ norms of integers in $\mathcal{O}$.
So, $\mathfrak{S}_d(1)$ contains all the positive definite integral unary Hermitian lattices with $g_d(1)=P(\mathcal{O})$.
[In fact, one has $g_d(1)=P(\mathcal{O})=2$ if $d=1,2,3,7,11$, $g_d(1)=P(\mathcal{O})=3$ if $d=19$, and $g_d(1)=P(\mathcal{O})=4$ if $d=43,67,163$.]

On the other hand, for imaginary quadratic fields with class number larger than $1$, \cite[Theorems 1 and 2]{jL2} determined the explicit form of $\mathfrak{S}_d(1)$ and the exact value of $g_d(1)$ for every $E=\mathbb{Q}\big(\sqrt{-d}\big)$ with class number $2$ or $3$.
[Notice that $g_d(1)=P(\mathcal{O})$ in all cases except for $E=\mathbb{Q}\big(\sqrt{-907}\big)$ where $g_{907}(1)=5>P(\mathcal{O})=4$.]
This work, as an essential extension to \cite{jL2}, provides an algorithm with rigorous mathematical proofs, which can determine the explicit form of $\mathfrak{S}_d(1)$ and the exact value of $g_d(1)$ for every imaginary quadratic field $E$ with any given class number.

This paper is organized as follows: In Section 2, I discuss the geometric language of Hermitian spaces and Hermitian lattices, in Section 3, I introduce an algorithm to determine $\mathfrak{S}_d(1)$ and $g_d(1)$ for all imaginary quadratic fields $E$ with rigorous proofs, and finally in Section 4, I provide SageMath codes and work out an example in details using these codes as a demonstration.

\section{Preliminaries}\label{Sec:Pre}
Throughout this paper, the notations and terminologies for lattices from the classical monograph of O'Meara \cite{otO} will be adopted.
For background information and terminologies specific to the Hermitian setting, one may consult the works of Shimura \cite{gS} and Gerstein \cite{ljG}.

Let $E=\mathbb{Q}\big(\sqrt{-d}\big)$ be an imaginary quadratic field for a square-free positive integer $d$, and let $\mathcal{O}$ be its ring of integers.
Then, one has $\mathcal{O}=\mathbb{Z}+\mathbb{Z}\omega$ with $\omega=\omega_d:=\sqrt{-d}$ if $d\equiv1,2~(\mathrm{mod}~4)$ and $\omega=\omega_d:=\big(1+\sqrt{-d}\big)/2$ if $d\equiv3~(\mathrm{mod}~4)$.
For each $\alpha\in E$, denote by $\overline\alpha$ the complex conjugate of $\alpha$ and define the norm of $\alpha$ to be $N(\alpha):=\alpha\overline\alpha$.
A Hermitian space $(V,h)$ is a vector space $V$ over $E$ that admits of a Hermitian map $h:V\times V\to E$ satisfying, for all $\alpha,\beta\in E$ and $v,v_1,v_2,w\in V$,
\begin{description}
\item[~~~~(0.1)] $h(v,w)=\overline{h(w,v)}$;
\item[~~~~(0.2)] $h(\alpha v_1+\beta v_2,w)=\alpha h(v_1,w)+\beta h(v_2,w)$.
\end{description}
Write $h(v):=h(v,v)$ for brevity.
By condition {\bf(0.1)}, one has $h(v)=\overline{h(v)}$, and thus, $h(v)\in\mathbb{Q}$ for every $v\in V$.

A Hermitian lattice $L$ is defined as a finitely generated $\mathcal{O}$-module in the Hermitian space $V$, and $L$ is said to be {\sl a lattice on} $V$ whenever $V=EL$.
We assume in the sequel that all Hermitian lattices $L$ are positive definite integral in the sense that $h(v)\in\mathbb{Z}_{>0}$ for all nonzero elements $v\in L$ and that $h(v,w)\in\mathcal{O}$ for all pairs of elements $v,w\in L$.
$L$ is said to be represented by another Hermitian lattice $K$, written as $L\to K$, provided there exists an injective linear map $\sigma:L\to K$ which preserves Hermitian maps, {\sl i.e.}, $h_L(v,w)=h_K(\sigma(v),\sigma(w))$ for all $v,w\in L$.
We thus call $\sigma$ a representation from $L$ to $K$.
In addition, $L$ and $K$ are said to be {\sl isometric} if $\sigma$ is bijective, and the isometry class containing $L$ is denoted by $\mathrm{cls}(L)$.

Let $V$ be an $n$-dimensional Hermitian space, and let $L$ be a Hermitian lattice on $V$; then, there are fractional $\mathcal{O}$-ideals $\mathfrak{A}_1,\ldots,\mathfrak{A}_n$ and a basis $\{v_1,\ldots,v_n\}$ of $V$ such that $L=\mathfrak{A}_1v_1+\cdots+\mathfrak{A}_nv_n$.
In particular, if we have $\mathfrak{A}_1=\cdots=\mathfrak{A}_n=\mathcal{O}$ for some basis $\{v_1,\ldots,v_n\}$, then we say $L$ a {\sl free} Hermitian lattice and associate to it a Gram matrix $M_L=(h(v_l,v_{l'}))$.
For example, a free unary Hermitian lattice $L=\mathcal{O}v$ with $h(v)=a$ can be identified with the Gram matrix $M_L=\langle a\rangle$, and $L\to I_m$ if and only if one can write $a$ as a sum of $m$ norms of integers in $\mathcal{O}$.

The {\sl scale} $\mathfrak{s}L$ of $L$ is the fractional $\mathcal{O}$-ideal generated by the set $\{h(v,w):v,w\in L\}$, and the {\sl volume} of $L$ is the fractional $\mathcal{O}$-ideal $\mathfrak{v}L:=\big(\mathfrak{A}_1\overline{\mathfrak{A}_1}\big)\cdots\big(\mathfrak{A}_n\overline{\mathfrak{A}_n}\big)\mathrm{det}(h(v_l,v_{l'}))$.
Each $\mathfrak{A}_j$ can be written as a product of integral powers of prime ideals in $\mathcal{O}$.
For each prime ideal $\mathfrak{P}$ in $\mathcal{O}$, denote by $p$ the prime number which lies below $\mathfrak{P}$.
Then, one has $\mathfrak{P}\overline{\mathfrak{P}}=p^2\mathcal{O}$ when $p$ is inert in $E$ and $\mathfrak{P}\overline{\mathfrak{P}}=p\mathcal{O}$ otherwise.
In consequence, there is a unique positive rational number $\delta_L$ (said to be the {\sl discriminant} of $L$) with $\mathfrak{v}L=\delta_L\mathcal{O}$.
Suppose $\mathcal{I}$ is a fractional $\mathcal{O}$-ideal.
$L$ is said to be $\mathcal{I}$-{\sl modular} if $\mathfrak{s}L=\mathcal{I}$ and $\mathfrak{v}L=\mathcal{I}^n$; in particular, $L$ is said to be {\sl unimodular} whenever $\mathcal{I}=\mathcal{O}$.

In \cite{jL2}, I determined the explicit form of $\mathfrak{S}_d(1)$ and the exact value of $g_d(1)$ for all $E=\mathbb{Q}\big(\sqrt{-d}\big)$ with class number $2$ or $3$ using the main strategy as follows.
Choose a representative $\mathfrak{P}$ for each ideal class in $E$: $\mathfrak{P}$ either is $\mathcal{O}$ or is a non-principal prime ideal in $\mathcal{O}$.
Every positive definite integral unary Hermitian lattice lies in the isometry class of $L^r=\mathfrak{P}v^r$ for a positive integer $r$.
When $\mathfrak{P}=\mathcal{O}$, then $h(v^r)=r$ and $L^r\cong\langle r\rangle$ is represented by $I_4$; so, $L^r\in\mathfrak{S}_d(1)$ for all positive integers $r$.
Otherwise, we know $\mathfrak{P}$ is a non-principal prime ideal with $\mathfrak{P}\overline{\mathfrak{P}}=p\mathcal{O}$ ($p$ lying below $\mathfrak{P}$) and $h(v^r)=r/p$.
Write $\mathfrak{P}=\mathcal{O}p+\mathcal{O}(s+t\omega)$ for some $s+t\omega\in\mathcal{O}$.
Using \cite[Lemma 2.1]{jL2}, $L^r$ can be represented by $I_m$ if and only if $r/p=\sum_{\ell=1}^mN(\gamma_\ell/p)$, where $\gamma_\ell:=a_\ell+b_\ell\omega\in\mathcal{O}$, with $a_\ell,b_\ell\in\mathbb{Z}$ for $1\leq\ell\leq m$ satisfying the conditions below
\begin{description}
\item[~~~~(1.1)] $p|(sa_\ell-dtb_\ell)$ and $p|(ta_\ell+sb_\ell)$ when $d\equiv1,2~(\mathrm{mod}~4)$;
\item[~~~~(1.2)] $p\big|\big(sa_\ell-\frac{d+1}{4}tb_\ell\big)$ and $p|(ta_\ell+(s+t)b_\ell)$ when $d\equiv3~(\mathrm{mod}~4)$.
\end{description}

Below, I extend this to imaginary quadratic fields $E=\mathbb{Q}\big(\sqrt{-d}\big)$ with class number $\mathtt{c}\geq4$.
As Lemma 2.2 in \cite{jL2} and Lemma \ref{Lem1} below lead to $4\leq g_d(1)\leq5$, $\mathfrak{S}_d(1)$ consists of all positive definite integral unary Hermitian lattices, except for those that cannot be represented by $I_5$.
For $g_d(1)$, I first observe that $L^r$ can be represented by $I_4$ when $r\geq C$ for a positive integer $C$ depending only on $d$ and $p$.
When $r<C$, we only need to check the representations $\mathfrak{P}v^r\to I_5$ and $\mathfrak{P}v^r\to I_4$: If every positive definite integral unary Hermitian lattice that can be represented by $I_5$ can also be represented by $I_4$, then $g_d(1)=4$; otherwise, $g_d(1)=5$.
I will design an algorithm to completely determine the explicit form of $\mathfrak{S}_d(1)$ and the exact value of $g_d(1)$.

The result below plays a pivotal role in the design of this algorithm, whose proof depends on the classical work of Mordell \cite{ljM}.

\begin{lem}\label{Lem1}
Let $E=\mathbb{Q}\big(\sqrt{-d}\big)$ be an imaginary quadratic field with $d$ a square-free positive integer, and let $\mathcal{O}$ be its ring of integers.
Let $\mathfrak{S}_d(1)$ be the set of all positive definite integral unary Hermitian lattices over $\mathcal{O}$ that can be represented by some $I_m$, and let $g_d(1)$ be the smallest positive integer such that all the lattices in $\mathfrak{S}_d(1)$ can be uniformly represented by $I_{g_d(1)}$.
Then, one has $g_d(1)\leq5$.
\end{lem}

\begin{proof}
Assume $L\in\mathfrak{S}_d(1)$ is a lattice represented by the Hermitian $\mathcal{O}$-lattice $I_m$ for an $m\in\mathbb{N}$.
Without loss of generality, suppose $L$ is a sublattice of $I_m$.
Then, we can write $L=\mathfrak{U}^{-1}v$ with $\mathfrak{U}$ an integral ideal and $v\subseteq\mathfrak{U}^m$, as $\mathfrak{U}^{-1}v\subseteq\mathcal{O}^m$.
Write $\mathfrak{U}=\mathbb{Z}\gamma_1+\mathbb{Z}\gamma_2$ for $\gamma_1,\gamma_2\in\mathfrak{U}$.
So, $v=\gamma_1w_1+\gamma_2w_2$ for two vectors $w_1,w_2\in\mathbb{Z}^m$.
Now, consider the quadratic $\mathbb{Z}$-lattice $L^{\mathbb{Z}}=\mathbb{Z}w_1+\mathbb{Z}w_2$, a rank (at most) $2$ sublattice of the quadratic $\mathbb{Z}$-lattice $I_m^{\mathbb{Z}}$, which clearly can be represented by $I_5^{\mathbb{Z}}$ noting Mordell \cite{ljM}.

We claim there is a representation $\varphi$ sending $L$ to $I_5$.
Let $w_1',w_2'$ be the images of $w_1,w_2$ in $I_5^{\mathbb{Z}}$, respectively.
Then, $v'=\gamma_1w_1'+\gamma_2w_2'\in\mathfrak{U}^5$ is a vector in the Hermitian $\mathcal{O}$-lattice $I_5$.
Set $\varphi(v)=v'$ to see, with $(~)^t$ being transpose,
\begin{equation}\label{Eq2.1}
\begin{aligned}
h(v)&=(\gamma_1w_1+\gamma_2w_2)^t\cdot\overline{(\gamma_1w_1+\gamma_2w_2)}\\
    &=(\gamma_1~~\gamma_2)\bigg(\begin{array}{c}
                                (w_1)^t \\
                                (w_2)^t
                                \end{array}\bigg)
      (w_1~~w_2)\bigg(\begin{array}{c}
                      \overline{\gamma_1} \\
                      \overline{\gamma_2}
                      \end{array}\bigg)\\
    &=(\gamma_1~~\gamma_2)\bigg(\begin{array}{c}
                                \big(w_1'\big)^t \\
                                \big(w_2'\big)^t
                                \end{array}\bigg)
      \big(w_1'~~w_2'\big)\bigg(\begin{array}{c}
                                \overline{\gamma_1} \\
                                \overline{\gamma_2}
                                \end{array}\bigg)\\
    &=\big(\gamma_1w_1'+\gamma_2w_2'\big)^t\cdot\overline{\big(\gamma_1w_1'+\gamma_2w_2'\big)}=H(v').
\end{aligned}
\end{equation}
We verified $h(v)=H(v')$ for Hermitian maps $h$ and $H$ on $L$ and $I_5$, respectively, and therefore, we have $L\to I_5$.
Consequently, $g_d(1)\leq5$.
\end{proof}

The next lemma says that for any two prime ideals in $\mathcal{O}$ which lie above the same prime number $p$, we only need to consider one of them.

\begin{lem}\label{Lem2}
Let $E=\mathbb{Q}\big(\sqrt{-d}\big)$ be an imaginary quadratic field with $d$ a square-free positive integer, let $p$ be a prime number splitting in $E$, and let $\mathfrak{P}_1,\mathfrak{P}_2$ be prime ideals in $\mathcal{O}$ lying above $p$.
Then, $r_1,r_2$ assume the same set of values in which $\mathfrak{P}_1v_1^{r_1},\mathfrak{P}_2v_2^{r_2}\in\mathfrak{S}_d(1)$, and there is a transformation between their representations by $I_m$.
\end{lem}

\begin{proof}
When $d\equiv3~(\mathrm{mod}~4)$, Proposition 8.3 in \cite{jN} leads to $\mathfrak{P}_1=\big(p,-\frac{n+1}{2}+\omega\big)$ and $\mathfrak{P}_2=\big(p,\frac{n-1}{2}+\omega\big)$ for the smallest positive odd integer $n$ with $-d\equiv n^2~(\mathrm{mod}~p)$; the proof of the stated result in Lemma \ref{Lem2} was given in \cite[Lemma 4.1]{jL2}.

Below, we discuss the case of $d\equiv1,2~(\mathrm{mod}~4)$.
Notice only odd prime numbers now split and $p\hspace{-0.8mm}\nmid\hspace{-0.8mm}d$.
By Theorem 25 in \cite{dM}, $\mathfrak{P}_1=(p,n+\omega)$ and $\mathfrak{P}_2=(p,n-\omega)$ are prime ideals in $\mathcal{O}$ lying above $p$ for the least positive integer $n$ with $-d\equiv n^2~(\mathrm{mod}~p)$.
Recall here that $\mathfrak{P}_jv_j^{r_j}\to I_m$ if and only if $r_j/p=\sum_{\ell=1}^mN(\gamma_\ell/p)$ for $j=1,2$, where $\gamma_\ell=a_\ell+b_\ell\omega\in\mathcal{O}$ with $a_\ell,b_\ell$ satisfying $p|(na_\ell-db_\ell)$ and $p|(a_\ell+nb_\ell)$ for $\mathfrak{P}_1$ and $p|(na_\ell+db_\ell)$ and $p|(a_\ell-nb_\ell)$ for $\mathfrak{P}_2$ by \cite[Lemma 2.1]{jL2}.
One easily verifies that $p|(a_\ell+nb_\ell)$ and $p|(a_\ell-nb_\ell)$ lead to $p|(na_\ell-db_\ell)$ and $p|(na_\ell+db_\ell)$, respectively.
So, when calculating the values of $r_1,r_2$ where $\mathfrak{P}_1v_1^{r_1},\mathfrak{P}_2v_2^{r_2}\in\mathfrak{S}_d(1)$, we only need the conditions $p|(a_\ell+nb_\ell)$ for $\mathfrak{P}_1$ and $p|(a_\ell-nb_\ell)$ for $\mathfrak{P}_2$.

For every pair of integers $a_\ell,b_\ell$ with $p|(a_\ell+nb_\ell)$, take $a'_\ell=a_\ell$ and $b'_\ell=-b_\ell$ to derive $p\big|\big(a'_\ell-nb'_\ell\big)$ and $N\Big(\frac{a'_\ell+b'_\ell\omega}{p}\Big)=N\Big(\frac{a_\ell+b_\ell\omega}{p}\Big)$, where $N\big(\frac{a+b\omega}{p}\big):=\frac{a^2}{p^2}+\frac{db^2}{p^2}$ here.
Conversely, choose $a_\ell=a'_\ell$ and $b_\ell=-b'_\ell$ for every pair of integers $a'_\ell,b'_\ell$ satisfying $p\big|\big(a'_\ell-nb'_\ell\big)$ to see $p|(a_\ell+nb_\ell)$ and $N\Big(\frac{a_\ell+b_\ell\omega}{p}\Big)=N\Big(\frac{a'_\ell+b'_\ell\omega}{p}\Big)$.
Thus, $r_1,r_2$ assume the same set of values in which $\mathfrak{P}_1v_1^{r_1},\mathfrak{P}_2v_2^{r_2}$ belong to $\mathfrak{S}_d(1)$, and $\mathfrak{P}_1v_1^{r_1},\mathfrak{P}_2v_2^{r_2}$ are represented by the same $I_m$ when $r_1=r_2$.
\end{proof}

\section{An algorithm}\label{Sec:Alg} 
Let $E=\mathbb{Q}\big(\sqrt{-d}\big)$ be an imaginary quadratic field with class number $\mathtt{c}\geq4$, and let $\mathcal{O}$ be its ring of integers.
When $\mathfrak{P}=\mathcal{O}$, then $\mathfrak{P}v^r\to I_4$ for all positive integers $r$.
So, we assume $\mathfrak{P}$ is a non-principal prime ideal in $\mathcal{O}$ and $p$ is the prime number lying below $\mathfrak{P}$, and prove in cases that $\mathfrak{P}v^r\to I_4$ provided $r$ is no less than a given integer $C$.
As Lemma 2.2 in \cite{jL2} and Lemma \ref{Lem1} here yield $4\leq g_d(1)\leq5$, we will utilize software to check the representations $\mathfrak{P}v^r\to I_4$ and $\mathfrak{P}v^r\to I_5$ when $r<C$.

Let $E(p)$ be the set of positive integers $r$ such that $\mathfrak{P}v^r$ cannot be represented by $I_5$, and write $g(p):=4$ if every unary Hermitian lattice $\mathfrak{P}v^r$ that can be represented by $I_5$ can also be represented by $I_4$; otherwise, write $g(p):=5$.
$\mathfrak{S}_d(1)$ consists of all the positive definite integral unary Hermitian lattices except for $\mathfrak{P}v^r$ with $r\in E(p)$, and $g_d(1)$ corresponds to the largest value of $g(p)$ when $p$ runs through all the primes lying below non-principal prime representatives of ideal classes of $E$.

\begin{thm}\label{Thm1}
Assume that $\mathfrak{P}$ is a non-principal prime ideal in $E$ and $p$ is the prime number lying below $\mathfrak{P}$.
When $p|d$, the unary Hermitian lattice $\mathfrak{P}v^r$ is represented by $I_4$ for every positive integer $r\geq C:=(p-1)d/p$.
\end{thm}

\begin{proof}
\noindent{\bf Case 1:} $d\equiv1,2~(\mathrm{mod}~4)$.
Then, $\mathcal{O}=\mathbb{Z}+\mathbb{Z}\omega$ for $\omega=\sqrt{-d}$.
By Theorem 25 in \cite{dM}, we have $\mathfrak{P}=\mathcal{O}p+\mathcal{O}\omega$.
Thus, it follows from \cite[Lemma 2.1]{jL2} that
\begin{equation*}
\omega\gamma_\ell=\omega(a_\ell+b_\ell\omega)=-db_\ell+a_\ell\omega
\end{equation*}
for each $1\leq\ell\leq4$, which leads to $p|a_\ell$ and $b_\ell$ arbitrary; so, one has
\begin{equation*}
h(v^{r})=\frac{r}{p}=\sum_{\ell=1}^4N\Big(\frac{\gamma_\ell}{p}\Big)
=\sum_{\ell=1}^4\bigg(\frac{a_\ell^2}{p^2}+\frac{db_\ell^2}{p^2}\bigg)=:\sum_{\ell=1}^4P_d(a_\ell,b_\ell).
\end{equation*}
Choose $a_\ell=p\tilde{a}_\ell$ with $b_\ell$ an arbitrary integer for $1\leq\ell\leq4$ to observe
\begin{equation}\label{Eq3.1}
r=p\sum_{\ell=1}^4P_d(a_\ell,b_\ell)=p\sum_{\ell=1}^4\bigg(\tilde{a}^2_\ell+\frac{db_\ell^2}{p^2}\bigg)
=\sum_{\ell=1}^4\Big(p\tilde{a}^2_\ell+\frac{d}{p}b_\ell^2\Big).
\end{equation}

\noindent{\bf Case 2:} $d\equiv 3~(\mathrm{mod}~4)$.
Then, $\mathcal{O}=\mathbb{Z}+\mathbb{Z}\omega$ for $\omega=\big(1+\sqrt{-d}\big)/2$.
By Theorem 25 in \cite{dM}, we get $\mathfrak{P}=\mathcal{O}p+\mathcal{O}\sqrt{-d}$.
So, it follows from \cite[Lemma 2.1]{jL2} that
\begin{equation*}
(-1+2\omega)\gamma_\ell=(-1+2\omega)(a_\ell+b_\ell\omega)=-\Big(a_\ell+\frac{d+1}{2}b_\ell\Big)+(2a_\ell+b_\ell)\omega
\end{equation*}
for each $1\leq\ell\leq 4$, which leads to $p|(2a_\ell+b_\ell)$; therefore, one has
\begin{equation*}
h(v^{r})=\frac{r}{p}=\sum_{\ell=1}^4N\Big(\frac{\gamma_\ell}{p}\Big)
=\sum_{\ell=1}^4\bigg(\frac{a_\ell^2}{p^2}+\frac{a_\ell b_\ell}{p^2}+\frac{(d+1)b_\ell^2}{4p^2}\bigg)=:\sum_{\ell=1}^4P_d(a_\ell,b_\ell).
\end{equation*}
Take $a_\ell=p\tilde{a}_\ell-\tilde{b}_\ell,b_\ell=2\tilde{b}_\ell$ with $\tilde{a}_\ell,\tilde{b}_\ell$ arbitrary integers for $1\leq\ell\leq4$ to observe
\begin{equation}\label{Eq3.2}
\begin{aligned}
r&=p\sum_{\ell=1}^4P_d(a_\ell,b_\ell)=p\sum_{\ell=1}^4\frac{1}{p^2}\Big(a_\ell^2+a_\ell b_\ell+\frac{d+1}{4}b_\ell^2\Big)\\
&=\sum_{\ell=1}^4\frac{1}{p}\big(\big(p\tilde{a}_\ell-\tilde{b}_\ell\big)^2+2\tilde{b}_\ell\big(p\tilde{a}_\ell-\tilde{b}_\ell\big)+(d+1)\tilde{b}^2_\ell\big)
=\sum_{\ell=1}^4\Big(p\tilde{a}^2_\ell+\frac{d}{p}\tilde{b}^2_\ell\Big).
\end{aligned}
\end{equation}

In equalities \eqref{Eq3.1} and \eqref{Eq3.2}, $d/p$ is an integer relatively prime to $p$, and therefore, $\{0,d/p,2d/p,\ldots,(p-1)d/p\}$ is a complete set of representatives of $\mathbb{Z}/p\mathbb{Z}$, with $r$ represented by $\displaystyle{p\sum_{\ell=1}^4\tilde{a}^2_\ell+\frac{d}{p}\sum_{\ell=1}^4\tilde{b}^2_\ell}$ for all $r\geq(p-1)d/p$.
\end{proof}

\begin{thm}\label{Thm2}
Suppose that $\mathfrak{P}$ is a non-principal prime ideal in $E$ and $2$ is the prime number lying below $\mathfrak{P}$.
When $d$ is odd, the unary Hermitian lattice $\mathfrak{P}v^r$ is represented by $I_4$ for every positive integer $r\geq C:=(d+1)/2$.
\end{thm}

\begin{proof}
As a matter of fact, Lemma 2.3 in \cite{jL2} implies that $\mathfrak{P}v^r\to I_4$ for all positive even integers $r$; the discussion below will lead to $\mathfrak{P}v^r\to I_3$ when $r\geq(d+1)/2$ is a positive odd integer.

\noindent{\bf Case 1:} $d\equiv1~(\mathrm{mod}~4)$.
Then, $\mathcal{O}=\mathbb{Z}+\mathbb{Z}\omega$ for $\omega=\sqrt{-d}$.
By \cite[Theorem 25]{dM}, one has $\mathfrak{P}=\mathcal{O}2+\mathcal{O}(1+\omega)$.
So, it follows from \cite[Lemma 2.1]{jL2} that
\begin{equation*}
(1+\omega)\gamma_\ell=(1+\omega)(a_\ell+b_\ell\omega)=(a_\ell-db_\ell)+(a_\ell+b_\ell)\omega
\end{equation*}
for each $1\leq\ell\leq 3$, which leads to $2|(a_\ell+b_\ell)$; therefore, we have
\begin{equation*}
h(v^{r})=\frac{r}{2}=\sum_{\ell=1}^3N\Big(\frac{\gamma_\ell}{2}\Big)
=\sum_{\ell=1}^3\bigg(\frac{a_\ell^2}{4}+\frac{db_\ell^2}{4}\bigg)=:\sum_{\ell=1}^3P_d(a_\ell,b_\ell).
\end{equation*}
Choose $a_1=2\tilde{a}_1+1,b_1=1$ and $a_\ell=2\tilde{a}_\ell,b_\ell=0$ for $2\leq\ell\leq3$ to observe
\begin{equation}\label{Eq3.3}
\begin{aligned}
r&=2\sum_{\ell=1}^3P_d(a_\ell,b_\ell)=\frac{1}{2}\big((2\tilde{a}_1+1)^2+d\big)+2\tilde{a}^2_2+2\tilde{a}^2_3\\
&=\frac{d+1}{2}+2\big(\tilde{a}^2_1+\tilde{a}_1+\tilde{a}^2_2+\tilde{a}^2_3\big).
\end{aligned}
\end{equation}
Noice $(d+1)/2$ is an odd integer and $2\big(\tilde{a}^2_1+\tilde{a}_1+\tilde{a}^2_2+\tilde{a}^2_3\big)$ can represent all positive even integers in view of \cite[Page 1368]{zS15}.
Therefore, $\mathfrak{P}v^r\to I_3$ when $r\geq(d+1)/2$ is a positive odd integer.

\noindent{\bf Case 2:} $d\equiv7~(\mathrm{mod}~8)$.
Then, $\mathcal{O}=\mathbb{Z}+\mathbb{Z}\omega$ for $\omega=\big(1+\sqrt{-d}\big)/2$.
By Theorem 25 in \cite{dM}, one has either $\mathfrak{P}=\mathcal{O}2+\mathcal{O}\omega$ or $\mathfrak{P}=\mathcal{O}2+\mathcal{O}(1-\omega)$.
By Lemma \ref{Lem2}, we take $\mathfrak{P}=\mathcal{O}2+\mathcal{O}\omega$ for discussion.
It follows from \cite[Lemma 2.1]{jL2} that
\begin{equation*}
\omega\gamma_\ell=\omega(a_\ell+b_\ell\omega)=-\frac{d+1}{4}b_\ell+(a_\ell+b_\ell)\omega
\end{equation*}
for each $1\leq\ell\leq 3$, which leads to $2|(a_\ell+b_\ell)$; therefore, we have
\begin{equation*}
h(v^{r})=\frac{r}{2}=\sum_{\ell=1}^3N\Big(\frac{\gamma_\ell}{2}\Big)
=\sum_{\ell=1}^3\bigg(\frac{a_\ell^2}{4}+\frac{a_\ell b_\ell}{4}+\frac{(d+1)b_\ell^2}{16}\bigg)=:\sum_{\ell=1}^3P_d(a_\ell,b_\ell).
\end{equation*}
Choose $a_\ell=2\tilde{a}_\ell+1$ and $b_\ell=-1$ for $1\leq\ell\leq3$ to observe
\begin{equation}\label{Eq3.4}
\begin{aligned}
r&=2\sum_{\ell=1}^3P_d(a_\ell,b_\ell)=\sum_{\ell=1}^3\frac{1}{2}\Big((2\tilde{a}_\ell+1)^2-(2\tilde{a}_\ell+1)+\frac{d+1}{4}\Big)\\
&=\sum_{\ell=1}^3\Big(2\tilde{a}_\ell^2+\tilde{a}_\ell+\frac{d+1}{8}\Big)=\frac{3(d+1)}{8}+\sum_{\ell=1}^3\tilde{a}_\ell(2\tilde{a}_\ell+1).
\end{aligned}
\end{equation}
Notice $3(d+1)/8$ is an integer and $\sum_{\ell=1}^3\tilde{a}_\ell(2\tilde{a}_\ell+1)$ represents all positive integers by \cite[Theorem 1.2]{zS17}.
Therefore, $\mathfrak{P}v^r\to I_3$ whenever $r\geq(d+1)/2\geq3(d+1)/8$ is a positive odd integer.

Note that when $d\equiv3~(\mathrm{mod}~8)$, the prime ideal $\mathfrak{P}$ lying above $2$ is $2\mathcal{O}$, a principal ideal, which is in the same ideal class of $\mathcal{O}$.
\end{proof}

\begin{thm}\label{Thm3}
Let $\mathfrak{P}$ be a non-principal prime ideal in $E$, and let $p$ be the prime number lying below $\mathfrak{P}$.
When $p$ is odd with $p\hspace{-0.8mm}\nmid\hspace{-0.8mm}d$, the unary Hermitian lattice $\mathfrak{P}v^r$ is represented by $I_4$ for every positive integer $r\geq C:=p(p-1)^2/4+(p-1)n+(d+n^2)/p+2pd$ if $d\equiv1,2~(\mathrm{mod}~4)$ and $r\geq C:=p(p-1)^2/4+(p-1)n+(d+n^2)/p+p(d+1)/4$ if $d\equiv3~(\mathrm{mod}~4)$.
\end{thm}

\begin{proof}
\noindent{\bf Case 1:} $d\equiv1,2~(\mathrm{mod}~4)$.
Then, $\mathcal{O}=\mathbb{Z}+\mathbb{Z}\omega$ for $\omega=\sqrt{-d}$.
By Theorem 25 in \cite{dM}, one has either $\mathfrak{P}=\mathcal{O}p+\mathcal{O}(n+\omega)$ or $\mathfrak{P}=\mathcal{O}p+\mathcal{O}(n-\omega)$, where $n$ is the least positive integer with $-d\equiv n^2~(\mathrm{mod}~p)$.
By Lemma \ref{Lem2}, we take $\mathfrak{P}=\mathcal{O}p+\mathcal{O}(n+\omega)$ for discussion.
It follows from \cite[Lemma 2.1]{jL2} that
\begin{equation*}
(n+\omega)\gamma_\ell=(n+\omega)(a_\ell+b_\ell\omega)=(na_\ell-db_\ell)+(a_\ell+nb_\ell)\omega
\end{equation*}
for each $1\leq\ell\leq 4$, which leads to $p|(a_\ell+nb_\ell)$; therefore, we have
\begin{equation*}
h(v^{r})=\frac{r}{p}=\sum_{\ell=1}^4N\Big(\frac{\gamma_\ell}{p}\Big)
=\sum_{\ell=1}^4\bigg(\frac{a_\ell^2}{p^2}+\frac{db_\ell^2}{p^2}\bigg)=:\sum_{\ell=1}^4P_d(a_\ell,b_\ell).
\end{equation*}
Choose $a_1=p\tilde{a}_1-n,b_1=1$ and $a_\ell=p\tilde{a}_\ell,b_\ell=p\tilde{b}_\ell$ for $2\leq\ell\leq4$ to observe
\begin{equation}\label{Eq3.5}
{\small\begin{aligned}
r=&~p\sum_{\ell=1}^4P_d(a_\ell,b_\ell)=\bigg(p\tilde{a}_1^2-2\tilde{a}_1n+\frac{d+n^2}{p}\bigg)+
p\big(\tilde{a}_2^2+\tilde{a}_3^2+\tilde{a}_4^2+d\tilde{b}_2^2+d\tilde{b}_3^2+d\tilde{b}_4^2\big)\\
:=&~S_1+pS_2.
\end{aligned}}
\end{equation}
For $S_1$, note $(d+n^2)/p$ is an integer and $p\tilde{a}_1^2-2\tilde{a}_1n$ forms a complete representative set when $\tilde{a}_1$ runs through the integral set $\{-(p-1)/2,\ldots,-1,0,1,\ldots,(p-1)/2\}$.
For $S_2$, recall $\tilde{a}_2^2+\tilde{a}_3^2+\tilde{a}_4^2$ can represent all positive integers except for those of the form $4^s(8t+7)$ with $s,t$ nonnegative integers.
Given an integer $m\geq2d$ of the form $4^s(8t+7)$, we know that either $m-d$ or $m-2d$ can be represented by $\tilde{a}_2^2+\tilde{a}_3^2+\tilde{a}_4^2$; details are listed below
{\small\begin{center}
\setlength{\arrayrulewidth}{0.2mm}
\setlength{\tabcolsep}{2pt}
\renewcommand{\arraystretch}{0.81}
\begin{tabular}{|m{2.5cm}|m{2cm}|m{2cm}|m{2cm}|m{2cm}|}
\hline           & \vskip2pt $d\equiv1~(\mathrm{mod}~8)$ & \vskip2pt $d\equiv2~(\mathrm{mod}~8)$ & \vskip2pt $d\equiv5~(\mathrm{mod}~8)$ & \vskip2pt $d\equiv6~(\mathrm{mod}~8)$\\
\hline \vskip2pt $m=4^0(8t+7)$
                 & \vskip2pt $m-d$                       & \vskip2pt $m-d$                       & \vskip2pt $m-d$                       & \vskip2pt $m-d$\\
\hline \vskip2pt $m=4^1(8t+7)$
                 & \vskip2pt $m-d$                       & \vskip2pt $m-d$                       & \vskip2pt $m-2d$                      & \vskip2pt $m-d$\\
\hline \vskip2pt $m=4^{s\geq2}(8t+7)$
                 & \vskip2pt $m-2d$                      & \vskip2pt $m-d$                       & \vskip2pt $m-d$                       & \vskip2pt $m-d$\\
\hline
\end{tabular}
\end{center}}
\noindent So, $S_2$ represents all the positive integers $m\geq2d$.
Combine $S_1$ and $S_2$ to see $\mathfrak{P}v^r\to I_4$ for every positive integer $r\geq p(p-1)^2/4+(p-1)n+(d+n^2)/p+2pd$.

\noindent{\bf Case 2:} $d\equiv3~(\mathrm{mod}~4)$.
Then, $\mathcal{O}=\mathbb{Z}+\mathbb{Z}\omega$ for $\omega=\big(1+\sqrt{-d}\big)/2$, and Proposition 8.3 in \cite{jN} leads to $\mathfrak{P}_1=\big(p,-\frac{n+1}{2}+\omega\big)$ and $\mathfrak{P}_2=\big(p,\frac{n-1}{2}+\omega\big)$ for the least positive odd integer $n$ with $-d\equiv n^2~(\mathrm{mod}~p)$.
By Lemma 2.1 in \cite{jL2} and Lemma \ref{Lem2}, we take $\mathfrak{P}=\mathcal{O}p+\mathcal{O}\big(\frac{n-1}{2}+\omega\big)$ for discussion and have
\begin{equation*}
\Big(\frac{n-1}{2}+\omega\Big)\gamma_\ell=\Big(\frac{n-1}{2}+\omega\Big)(a_\ell+b_\ell\omega)
=\Big(\frac{n-1}{2}a_\ell-\frac{d+1}{4}b_\ell\Big)+\Big(a_\ell+\frac{n+1}{2}b_\ell\Big)\omega
\end{equation*}
for each $1\leq\ell\leq 4$, which leads to $p\big|\big(a_\ell+\frac{n+1}{2}b_\ell\big)$; therefore, one has
\begin{equation*}
h(v^{r})=\frac{r}{p}=\sum_{\ell=1}^4N\Big(\frac{\gamma_\ell}{p}\Big)
=\sum_{\ell=1}^4\bigg(\frac{a_\ell^2}{p^2}+\frac{a_\ell b_\ell}{p^2}+\frac{(d+1)b_\ell^2}{4p^2}\bigg)=:\sum_{\ell=1}^4P_d(a_\ell,b_\ell).
\end{equation*}
Take $a_1=p\tilde{a}_1-(n+1),b_1=2$, $a_2=p\tilde{a}_2,b_2=p$, and $a_\ell=p\tilde{a}_\ell,b_\ell=0$ for $3\leq\ell\leq4$ to observe
\begin{equation}\label{Eq3.6}
{\small\begin{aligned}
r=&~p\sum_{\ell=1}^4P_d(a_\ell,b_\ell)=\bigg(p\tilde{a}_1^2-2\tilde{a}_1n+\frac{d+n^2}{p}\bigg)+p\Big(\tilde{a}_2^2+\tilde{a}_2+\frac{d+1}{4}+\tilde{a}_3^2+\tilde{a}_4^2\Big)\\
:=&~S_1+pS_2.
\end{aligned}}
\end{equation}
For $S_1$, note $(d+n^2)/p$ is an integer and $p\tilde{a}_1^2-2\tilde{a}_1n$ forms a complete representative set when $\tilde{a}_1$ runs through the integral set $\{-(p-1)/2,\ldots,-1,0,1,\ldots,(p-1)/2\}$.
For $S_2$, recall $\tilde{a}_2^2+\tilde{a}_2+\tilde{a}_3^2+\tilde{a}_4^2$ can represent all positive integers by \cite[Page 1368]{zS15}.
Thus, $S_2$ represents all the positive integers $m\geq(d+1)/4$.
Combine $S_1$ and $S_2$ to see $\mathfrak{P}v^r\to I_4$ for every positive integer $r\geq p(p-1)^2/4+(p-1)n+(d+n^2)/p+p(d+1)/4$.
\end{proof}

\begin{algorithm}
\renewcommand{\thealgorithm}{}
\caption{An algorithm for $E(p)$ and $g(p)$}
{\bf Input:}  $p$, $d$ and other indispensable parameters
\vskip 0pt\noindent {\bf Output:} $E(p)$ and $g(p)$
\begin{algorithmic}[1]
    \State Rewrite $P_d(a_\ell,b_\ell)$ and remove the congruence restrictions on $a_\ell$ and $b_\ell$; see Section 4 for details.
    \State $E(p)\gets$ the set of values of $r$ that are less than $C$ and that cannot be represented by $\sum_{\ell=1}^5pP_d(a_\ell, b_\ell)$.
    \State $F(p)\gets$ the set of values of $r$ that are less than $C$ and that cannot be represented by $\sum_{\ell=1}^4pP_d(a_\ell, b_\ell)$.
    \If{$E(p)=F(p)$}
    \State $g(p)=4$
    \Else
    \State $g(p)=5$
    \EndIf
    \State Return $g(p),E(p)$
\end{algorithmic}
\end{algorithm}

\section{SageMath codes for an example}\label{Sec:Ex}
In this section, I provide the SageMath codes for this algorithm and use the imaginary quadratic field $E=\mathbb{Q}\big(\sqrt{-87}\big)$ (with class number $6$) as a demonstration for the process of determining the set $\mathfrak{S}_d(1)$ and the value of the $g$-invariant $g_d(1)$.

In order to determine the positive integers that can be represented by the quadratic forms $\sum_{\ell=1}^5pP_d(a_\ell,b_\ell)$ or $\sum_{\ell=1}^4pP_d(a_\ell,b_\ell)$, one needs to remove the restrictions on $a_\ell,b_\ell$ using the given congruence relation between them; for instance, when $p\hspace{-0.8mm}\nmid\hspace{-0.8mm}d$ and $d\equiv3~(\mathrm{mod}~4)$, we have
\begin{equation*}
pP_d(a_\ell,b_\ell)=\frac{1}{p}\Big(a_\ell^2+a_\ell b_\ell+\frac{d+1}{4}b_\ell^2\Big)
\end{equation*}
with $p\big|\big(a_\ell+\frac{n+1}{2}b_\ell\big)$. Choose $a_\ell=p\tilde{a}_\ell-\frac{n+1}{2}\tilde{b}_\ell,b_\ell=\tilde{b}_\ell$ for arbitrary integers $\tilde{a}_\ell,\tilde{b}_\ell$ to deduce
\begin{equation*}
pP_d(a_\ell,b_\ell)=p\tilde{a}_\ell^2-n\tilde{a}_\ell\tilde{b}_\ell+\frac{d+n^2}{4p}\tilde{b}_\ell^2,
\end{equation*}
a quadratic form with integral variables $\tilde{a}_\ell,\tilde{b}_\ell$.
This new expression of $pP_d(a_\ell,b_\ell)$ is used in Code $6$.

The followings are the codes for the six cases discussed above.

\newpage
\begin{center}
{\bf Code 1: $p|d$ and $d\equiv1,2~(\mathrm{mod}~4)$}
\end{center}
\noindent {\bf Input:} $p;d$
\vskip 0pt\noindent {\bf Output:} $g(p);E(p)$
\begin{lstlisting}[language=Python]
sage: p=?; d=?; C=(p-1)*d/p
sage: Q=QuadraticForm(ZZ, 10, [p,0,0,0,0,0,0,0,0,0,d/p,0,0,0,0,0,0,0,0,p,0,0,0,0,0,0,0,d/p,0,0,0,0,0,0,p,0,0,0,0,0,d/p,0,0,0,0,p,0,0,0,d/p,0,0,p,0,d/p])
sage: S=Q.representation_number_list(C)
sage: Q=QuadraticForm(ZZ, 8, [p,0,0,0,0,0,0,0,d/p,0,0,0,0,0,0,p,0,0,0,0,0,d/p,0,0,0,0,p,0,0,0,d/p,0,0,p,0,d/p])
sage: T=Q.representation_number_list(C)
sage: def u(l):
sage:     if S[l]==0:return l
sage:     else:return 0
sage: E=[u(l) for l in [0..C-1]]
sage: E(p)=[value for value in E if value !=0]
sage: def v(l):
sage:     if T[l]==0:return l
sage:     else:return 0
sage: F=[v(l) for l in [0..C-1]]
sage: F(p)=[value for value in F if value !=0]
sage: def g(p):
sage:     if E(p)==F(p): return 4
sage:     else: return 5
sage: g(p);E(p)
\end{lstlisting}

\newpage
\begin{center}
{\bf Code 2: $p|d$ and $d\equiv3~(\mathrm{mod}~4)$}
\end{center}
\noindent {\bf Input:} $p;d$
\vskip 0pt\noindent {\bf Output:} $g(p);E(p)$
\begin{lstlisting}[language=Python]
sage: p=?; d=?; C=(p-1)*d/p
sage: Q=QuadraticForm(ZZ, 10, [d/p,-d,0,0,0,0,0,0,0,0,p*(1+d)/4,0,0,0,0,0,0,0,0,d/p,-d,0,0,0,0,0,0,p*(1+d)/4,0,0,0,0,0,0,d/p,-d,0,0,0,0,p*(1+d)/4,0,0,0,0,d/p,-d,0,0,p*(1+d)/4,0,0,d/p,-d,p*(1+d)/4])
sage: S=Q.representation_number_list(C)
sage: Q=QuadraticForm(ZZ, 8, [d/p,-d,0,0,0,0,0,0,p*(1+d)/4,0,0,0,0,0,0,d/p,-d,0,0,0,0,p*(1+d)/4,0,0,0,0,d/p,-d,0,0,p*(1+d)/4,0,0,d/p,-d,p*(1+d)/4])
sage: T=Q.representation_number_list(C)
sage: def u(l):
sage:     if S[l]==0:return l
sage:     else:return 0
sage: E=[u(l) for l in [0..C-1]]
sage: E(p)=[value for value in E if value !=0]
sage: def v(l):
sage:     if T[l]==0:return l
sage:     else:return 0
sage: F=[v(l) for l in [0..C-1]]
sage: F(p)=[value for value in F if value !=0]
sage: def g(p):
sage:     if E(p)==F(p): return 4
sage:     else: return 5
sage: g(p);E(p)
\end{lstlisting}

\newpage
\begin{center}
{\bf Code 3: $2\hspace{-0.8mm}\nmid\hspace{-0.8mm}d$ and $d\equiv1~(\mathrm{mod}~4)$}
\end{center}
\noindent {\bf Input:} $d$
\vskip 0pt\noindent {\bf Output:} $g(p);E(p)$
\begin{lstlisting}[language=Python]
sage: p=2; d=?; C=(1+d)/2
sage: Q=QuadraticForm(ZZ, 10, [2,-2,0,0,0,0,0,0,0,0,(1+d)/2,0,0,0,0,0,0,0,0,2,-2,0,0,0,0,0,0,(1+d)/2,0,0,0,0,0,0,2,-2,0,0,0,0,(1+d)/2,0,0,0,0,2,-2,0,0,(1+d)/2,0,0,2,-2,(1+d)/2])
sage: S=Q.representation_number_list(C)
sage: Q=QuadraticForm(ZZ, 8, [2,-2,0,0,0,0,0,0,(1+d)/2,0,0,0,0,0,0,2,-2,0,0,0,0,(1+d)/2,0,0,0,0,2,-2,0,0,(1+d)/2,0,0,2,-2,(1+d)/2])
sage: T=Q.representation_number_list(C)
sage: def u(l):
sage:     if S[l]==0:return l
sage:     else:return 0
sage: E=[u(l) for l in [0..C-1]]
sage: E(p)=[value for value in E if value !=0]
sage: def v(l):
sage:     if T[l]==0:return l
sage:     else:return 0
sage: F=[v(l) for l in [0..C-1]]
sage: F(p)=[value for value in F if value !=0]
sage: def g(p):
sage:     if E(p)==F(p): return 4
sage:     else: return 5
sage: g(p);E(p)
\end{lstlisting}

\newpage
\begin{center}
{\bf Code 4: $2\hspace{-0.8mm}\nmid\hspace{-0.8mm}d$ and $d\equiv7~(\mathrm{mod}~8)$}
\end{center}
\noindent {\bf Input:} $d$
\vskip 0pt\noindent {\bf Output:} $g(p);E(p)$
\begin{lstlisting}[language=Python]
sage: p=2; d=?; C=(1+d)/2
sage: Q=QuadraticForm(ZZ, 10, [2,-1,0,0,0,0,0,0,0,0,(1+d)/8,0,0,0,0,0,0,0,0,2,-1,0,0,0,0,0,0,(1+d)/8,0,0,0,0,0,0,2,-1,0,0,0,0,(1+d)/8,0,0,0,0,2,-1,0,0,(1+d)/8,0,0,2,-1,(1+d)/8])
sage: S=Q.representation_number_list(C)
sage: Q=QuadraticForm(ZZ, 8, [2,-1,0,0,0,0,0,0,(1+d)/8,0,0,0,0,0,0,2,-1,0,0,0,0,(1+d)/8,0,0,0,0,2,-1,0,0,(1+d)/8,0,0,2,-1,(1+d)/8])
sage: T=Q.representation_number_list(C)
sage: def u(l):
sage:     if S[l]==0:return l
sage:     else:return 0
sage: E=[u(l) for l in [0..C-1]]
sage: E(p)=[value for value in E if value !=0]
sage: def v(l):
sage:     if T[l]==0:return l
sage:     else:return 0
sage: F=[v(l) for l in [0..C-1]]
sage: F(p)=[value for value in F if value !=0]
sage: def g(p):
sage:     if E(p)==F(p): return 4
sage:     else: return 5
sage: g(p);E(p)
\end{lstlisting}

\newpage
\begin{center}
{\bf Code 5: $p\hspace{-0.8mm}\nmid\hspace{-0.8mm}d$ and $d\equiv1,2~(\mathrm{mod}~4)$}
\end{center}
\noindent {\bf Input:} $p;d;n$
\vskip 0pt\noindent {\bf Output:} $g(p);E(p)$
\begin{lstlisting}[language=Python]
sage: p=?; d=?; n=?; C=p*(p-1)*(p-1)/4+(p-1)*n+(d+n*n)/p+2*p*d
sage: Q=QuadraticForm(ZZ, 10, [p,-2*n,0,0,0,0,0,0,0,0,(d+n*n)/p,0,0,0,0,0,0,0,0,p,-2*n,0,0,0,0,0,0,(d+n*n)/p,0,0,0,0,0,0,p,-2*n,0,0,0,0,(d+n*n)/p,0,0,0,0,p,-2*n,0,0,(d+n*n)/p,0,0,p,-2*n,(d+n*n)/p])
sage: S=Q.representation_number_list(C)
sage: Q=QuadraticForm(ZZ, 8, [p,-2*n,0,0,0,0,0,0,(d+n*n)/p,0,0,0,0,0,0,p,-2*n,0,0,0,0,(d+n*n)/p,0,0,0,0,p,-2*n,0,0,(d+n*n)/p,0,0,p,-2*n,(d+n*n)/p])
sage: T=Q.representation_number_list(C)
sage: def u(l):
sage:     if S[l]==0:return l
sage:     else:return 0
sage: E=[u(l) for l in [0..C-1]]
sage: E(p)=[value for value in E if value !=0]
sage: def v(l):
sage:     if T[l]==0:return l
sage:     else:return 0
sage: F=[v(l) for l in [0..C-1]]
sage: F(p)=[value for value in F if value !=0]
sage: def g(p):
sage:     if E(p)==F(p): return 4
sage:     else: return 5
sage: g(p);E(p)
\end{lstlisting}

\newpage
\begin{center}
{\bf Code 6: $p\hspace{-0.8mm}\nmid\hspace{-0.8mm}d$ and $d\equiv3~(\mathrm{mod}~4)$}
\end{center}
\noindent {\bf Input:} $p;d;n$
\vskip 0pt\noindent {\bf Output:} $g(p);E(p)$
\begin{lstlisting}[language=Python]
sage: p=?; d=?; n=?; C=p*(p-1)*(p-1)/4+(p-1)*n+(d+n*n)/p+p*(d+1)/4
sage: Q=QuadraticForm(ZZ, 10, [p,-n,0,0,0,0,0,0,0,0,(d+n*n)/(4*p),0,0,0,0,0,0,0,0,p,-n,0,0,0,0,0,0,(d+n*n)/(4*p),0,0,0,0,0,0,p,-n,0,0,0,0,(d+n*n)/(4*p),0,0,0,0,p,-n,0,0,(d+n*n)/(4*p),0,0,p,-n,(d+n*n)/(4*p)])
sage: S=Q.representation_number_list(C)
sage: Q=QuadraticForm(ZZ, 8, [p,-n,0,0,0,0,0,0,(d+n*n)/(4*p),0,0,0,0,0,0,p,-n,0,0,0,0,(d+n*n)/(4*p),0,0,0,0,p,-n,0,0,(d+n*n)/(4*p),0,0,p,-n,(d+n*n)/(4*p)])
sage: T=Q.representation_number_list(C)
sage: def u(l):
sage:     if S[l]==0:return l
sage:     else:return 0
sage: E=[u(l) for l in [0..C-1]]
sage: E(p)=[value for value in E if value !=0]
sage: def v(l):
sage:     if T[l]==0:return l
sage:     else:return 0
sage: F=[v(l) for l in [0..C-1]]
sage: F(p)=[value for value in F if value !=0]
sage: def g(p):
sage:     if E(p)==F(p): return 4
sage:     else: return 5
sage: g(p);E(p)
\end{lstlisting}

\newpage
Now, consider the imaginary quadratic field $E=\mathbb{Q}\big(\sqrt{-87}\big)$ with class number $6$.
\vskip10pt
\noindent{\bf Step 1.} Identify $5$ prime ideas as representatives of the $5$ non-principal ideal classes.
\vskip 0pt\noindent {\bf Input:} $d$
\vskip 0pt\noindent {\bf Output:} $6$ nonequivalent prime ideal representatives.
\begin{lstlisting}[language=Python]
sage: d=87
sage: K.<a>=NumberField(x^2+d)
sage: K.class_number()
6
sage: Cl=K.class_group()
sage: [c.representative_prime() for c in Cl]
[Fractional ideal (5),
 Fractional ideal (2,1/2*a-1/2),
 Fractional ideal (7,1/2*a+5/2),
 Fractional ideal (3,1/2*a+3/2),
 Fractional ideal (7,1/2*a+9/2),
 Fractional ideal (2,1/2*a+1/2)]
\end{lstlisting}
$a=\sqrt{-87}$ in this SageMath code outcome.
Denote by $\mathfrak{P}_1=\mathcal{O}$, $\mathfrak{P}_2=\mathcal{O}2+\mathcal{O}(-1+\omega)$, $\mathfrak{P}_3=\mathcal{O}7+\mathcal{O}(2+\omega)$, $\mathfrak{P}_4=\mathcal{O}3+\mathcal{O}(1+\omega)$, $\mathfrak{P}_5=\mathcal{O}7+\mathcal{O}(4+\omega)$, and $\mathfrak{P}_6=\mathcal{O}2+\mathcal{O}\omega$.

\newpage
\noindent{\bf Step 2:} From the above result, $3$ is ramified in $E$ and $2,7$ splits in $E$.
By Lemma \ref{Lem2}, one only needs to check the representations of $\mathfrak{P}_2v_2^{r_2}$, $\mathfrak{P}_3v_3^{r_3}$ and $\mathfrak{P}_4v_4^{r_4}$ by $I_m$.
\vskip 0pt\noindent {\bf Input:} $d=87$
\vskip 0pt\noindent {\bf Output:} $g(2);E(2)$
\begin{lstlisting}[language=Python]
sage: p=2; d=87; C=(1+d)/2
sage: Q=QuadraticForm(ZZ, 10, [2,-1,0,0,0,0,0,0,0,0,(1+d)/8,0,0,0,0,0,0,0,0,2,-1,0,0,0,0,0,0,(1+d)/8,0,0,0,0,0,0,2,-1,0,0,0,0,(1+d)/8,0,0,0,0,2,-1,0,0,(1+d)/8,0,0,2,-1,(1+d)/8])
sage: S=Q.representation_number_list(C)
sage: Q=QuadraticForm(ZZ, 8, [2,-1,0,0,0,0,0,0,(1+d)/8,0,0,0,0,0,0,2,-1,0,0,0,0,(1+d)/8,0,0,0,0,2,-1,0,0,(1+d)/8,0,0,2,-1,(1+d)/8])
sage: T=Q.representation_number_list(C)
sage: def u(l):
sage:     if S[l]==0:return l
sage:     else:return 0
sage: E=[u(l) for l in [0..C-1]]
sage: E(p)=[value for value in E if value !=0]
sage: def v(l):
sage:     if T[l]==0:return l
sage:     else:return 0
sage: F=[v(l) for l in [0..C-1]]
sage: F(p)=[value for value in F if value !=0]
sage: def g(p):
sage:     if E(p)==F(p): return 4
sage:     else: return 5
sage: g(p);E(p)
4
(1, 3, 5, 7, 9)
\end{lstlisting}

\newpage
\noindent {\bf Input:} $p=7;d=87;n=5$
\vskip 0pt\noindent {\bf Output:} $g(7);E(7)$
\begin{lstlisting}[language=Python]
sage: p=7; d=87; n=5; C=p*(p-1)*(p-1)/4+(p-1)*n+(d+n*n)/p+p*(d+1)/4
sage: Q=QuadraticForm(ZZ, 10, [p,-n,0,0,0,0,0,0,0,0,(d+n*n)/(4*p),0,0,0,0,0,0,0,0,p,-n,0,0,0,0,0,0,(d+n*n)/(4*p),0,0,0,0,0,0,p,-n,0,0,0,0,(d+n*n)/(4*p),0,0,0,0,p,-n,0,0,(d+n*n)/(4*p),0,0,p,-n,(d+n*n)/(4*p)])
sage: S=Q.representation_number_list(C)
sage: Q=QuadraticForm(ZZ, 8, [p,-n,0,0,0,0,0,0,(d+n*n)/(4*p),0,0,0,0,0,0,p,-n,0,0,0,0,(d+n*n)/(4*p),0,0,0,0,p,-n,0,0,(d+n*n)/(4*p),0,0,p,-n,(d+n*n)/(4*p)])
sage: T=Q.representation_number_list(C)
sage: def u(l):
sage:     if S[l]==0:return l
sage:     else:return 0
sage: E=[u(l) for l in [0..C-1]]
sage: E(p)=[value for value in E if value !=0]
sage: def v(l):
sage:     if T[l]==0:return l
sage:     else:return 0
sage: F=[v(l) for l in [0..C-1]]
sage: F(p)=[value for value in F if value !=0]
sage: def g(p):
sage:     if E(p)==F(p): return 4
sage:     else: return 5
sage: g(p);E(p)
4
(1, 2, 3, 5, 9)
\end{lstlisting}

\newpage
\noindent {\bf Input:} $p=3;d=87$
\vskip 0pt\noindent {\bf Output:} $g(3);E(3)$
\begin{lstlisting}[language=Python]
sage: p=3; d=87; C=(p-1)*d/p
sage: Q=QuadraticForm(ZZ, 10, [d/p,-d,0,0,0,0,0,0,0,0,p*(1+d)/4,0,0,0,0,0,0,0,0,d/p,-d,0,0,0,0,0,0,p*(1+d)/4,0,0,0,0,0,0,d/p,-d,0,0,0,0,p*(1+d)/4,0,0,0,0,d/p,-d,0,0,p*(1+d)/4,0,0,d/p,-d,p*(1+d)/4])
sage: S=Q.representation_number_list(C)
sage: Q=QuadraticForm(ZZ, 8, [d/p,-d,0,0,0,0,0,0,p*(1+d)/4,0,0,0,0,0,0,d/p,-d,0,0,0,0,p*(1+d)/4,0,0,0,0,d/p,-d,0,0,p*(1+d)/4,0,0,d/p,-d,p*(1+d)/4])
sage: T=Q.representation_number_list(C)
sage: def u(l):
sage:     if S[l]==0:return l
sage:     else:return 0
sage: E=[u(l) for l in [0..C-1]]
sage: E(p)=[value for value in E if value !=0]
sage: def v(l):
sage:     if T[l]==0:return l
sage:     else:return 0
sage: F=[v(l) for l in [0..C-1]]
sage: F(p)=[value for value in F if value !=0]
sage: def g(p):
sage:     if E(p)==F(p): return 4
sage:     else: return 5
sage: g(p);E(p)
4
(1, 2, 4, 5, 7, 10, 13)
\end{lstlisting}

The preceding results together yield $g_{87}(1)=4$ and $\mathfrak{S}_{87}(1)=\big\{\mathcal{O}v_1^{r_1}:r_1\geq1\big\}\bigcup\allowbreak
\big\{\mathfrak{P}_2v_2^{r_2}:r_2\neq1,3,5,7,9\big\}\bigcup\big\{\mathfrak{P}_3v_3^{r_3}:r_3\neq1,2,3,5,9\big\}
\bigcup\big\{\mathfrak{P}_4v_4^{r_4}:r_4\neq1,2,4,5,7,\allowbreak
10,13\big\}\bigcup\big\{\mathfrak{P}_5v_5^{r_5}:r_5\neq1,2,3,5,9\big\}\bigcup\big\{\mathfrak{P}_6v_6^{r_6}:r_6\neq1,3,5,7,9\big\}$.


\end{document}